\documentclass[a4paper,12pt,leqno]{article}

\usepackage{amssymb,amsmath}
\usepackage{amsthm}
\usepackage[abbrev]{amsrefs}
\usepackage{xy}
\xyoption{all}
\usepackage{graphicx}

\numberwithin{equation}{section}
\newtheorem{theorem}{Theorem}[section]
\newtheorem{lemma}[theorem]{Lemma}
\newtheorem{corollary}[theorem]{Corollary}
\newtheorem{exm}[theorem]{Example}
\newenvironment{example}{\begin{exm}\em}{\end{exm}}

\newcommand\V{\bigvee}
\newcommand\Max{\operatorname{Max}}
\newcommand\ie{i.e.}
\newcommand\st{\mid}
\newcommand\opens{\operatorname{\mathcal{O}}}
\newcommand\groupoid{\operatorname{\mathcal{G}}}
\newcommand\downsegment{{\downarrow}}
\newcommand\ts{\mathrm{T}}
\newcommand\ipi{\mathcal I}
\newcommand\lcc{\operatorname{{\mathcal L}^{\vee}}}

\newcommand\dom{\operatorname{dom}}

\begin{document}

\title{The many groupoids of a stably Gelfand quantale\thanks{Work funded by FCT/Portugal through project PEst-OE/EEI/LA0009/2013 and by COST (European Cooperation in Science and Technology) through COST Action MP1405 QSPACE.}}
\author{{\sc Pedro Resende}}

\date{~}

\maketitle

\vspace*{-1cm}
\begin{abstract}
We study the projections of an arbitrary stably Gelfand quantale $Q$ and show that each projection determines a pseudogroup $S\subset Q$ (and a corresponding localic \'etale groupoid $G$) together with a map of involutive quantales $p:Q\to\lcc(S)\ [=\opens(G)]$. As an application we obtain a simplified axiomatization of inverse quantal frames (= quantales of \'etale groupoids) whereby such a quantale is shown to be the same as a unital stably Gelfand quantal frame $Q$ whose partial units cover $Q$.
\\
\vspace*{-2mm}~
\vspace*{-2mm}~
\end{abstract}

\setcounter{tocdepth}{2}
\tableofcontents

\section{Introduction}

Gelfand quantales are the involutive quantales $Q$ that satisfy the condition $aa^*a=a$ for all right-sided elements $a\in Q$, a motivating example being the involutive quantale $\Max A$ of a C*-algebra $A$~\cite{Curacao,MP1,MP2,KPRR,KR}. The stronger notion of stably Gelfand quantale requires the condition $aa^*a=a$ to be satisfied by every element $a\in Q$ for which $aa^*a\le a$. This terminology appeared in~\cite{GSQS}, where it was also noted that one pleasant property of such quantales is that several variants of the notion of quantale-valued set collapse into a single one, in this sense making stably Gelfand quantales the most general class of involutive quantales for which it seems to make sense to study notions of sheaf (an earlier appearance of the stably Gelfand law, in the context of sheaves on quantaloids, is given by the pseudo-rightsided quantaloids of~\cite{Garraway}).

A subclass consists of the strongly Gelfand quantales, which are the involutive quantales $Q$ for which the law $a\le aa^*a$ holds for all $a\in Q$. Examples of strongly Gelfand quantales are the involutive quantales $\opens(G)$ of \'etale or open groupoids $G$~\cite{Re07,PR12}. In particular, the quantales $\opens(G)$ of \'etale groupoids $G$ are the inverse quantal frames, which are examples of supported quantales. The latter are unital and also strongly Gelfand, and one crucial property of any supported quantale $Q$ is that the set of partial units
\begin{equation}\label{pu}
\ipi(Q)=\{a\in Q\st a^*a\le e,\ aa^*\le e\}
\end{equation}
is a complete and infinitely distributive inverse semigroup, herein referred to just as a pseudogroup, whose multiplication and natural order coincide with those of $Q$. Then the inverse quantal frames $Q$ can be precisely characterized as being the supported quantales (with stable support) whose order satisfies the distributivity property of locales, such that the partial units cover $Q$~\cite{Re07}:
\begin{equation}\label{cc}
\V\ipi(Q)=1\;.
\end{equation}

A pseudogroup can also be obtained from an arbitrary Gelfand quantale, in the following specific sense: the subquantale $\ts(Q)$ of two-sided elements of any Gelfand quantale $Q$ is a locale, hence the same as a pseudogroup whose elements are all idempotents:
\begin{equation}\label{ts}
\ts(Q) = \{a\in Q\st a1\le a,\ 1a\le a\}\;.
\end{equation}

By a projection of a Gelfand quantale $Q$ is meant an element $b\in Q$ such that $b^2=b=b^*$ (equivalently, $b^2\le b\le b^*$). An example is the top element $1$ (and indeed any two-sided element) and, in the case of unital quantales, the multiplication unit $e$.
In this paper we unify the above two examples of pseudogroups by considering elements which are both like partial units and two-sided elements with respect to a projection $b\in Q$:
\[
\ipi_b(Q)=\{a\in Q\st a^*a\le b,\ aa^*\le b,\ ab\le a,\ ba\le a\}\;.
\]
As it will turn out, if $Q$ is stably Gelfand then $\ipi_b(Q)$ is a pseudogroup (Theorem~\ref{thm:IbQ}). The purpose of this paper is to study the projections of stably Gelfand quantales in terms of their pseudogroups (and associated \'etale groupoids). As an application, we shall obtain a simplified characterization of inverse quantal frames, which in particular does not depend on the theory of supported quantales (Corollary~\ref{cor:iqfs}). Namely, we prove that an involutive quantale $Q$ is an inverse quantal frame if and only the following conditions are satisfied:
\begin{enumerate}
\item $Q$ is stably Gelfand;
\item $Q$ is unital;
\item $Q$ is a quantal frame;
\item The covering condition \eqref{cc} holds.
\end{enumerate}

Another motivation behind the notion of stably Gelfand quantale is the quantale $\Max A$ of a C*-algebra $A$. This is because a projection $B\in\Max A$ is exactly a sub-C*-algebra $B\subset A$, and the pseudogroup $\ipi_B(\Max A)$ is closely related to inverse semigroups and groupoids whose construction from certain subalgebras is well known~\cite{Kumjian,Renault,Exel}. The appropriate context in which to study this relation is that of Fell bundles on groupoids~\cite{Kumjian98} (or inverse semigroups).
We study the interplay between quantales and Fell bundles in~\cite{QFB}, where also the proof that $\Max A$ is stably Gelfand is presented. Nevertheless the present paper contains a quantale-theoretic analogue of the relation between sub-C*-algebras and Fell bundles, namely in the form of a relation between projections of a stably Gelfand quantale $Q$ and maps of involutive quantales $p:Q\to\opens(G)$, where a map is dually defined to be a homomorphism $p^*:\opens(G)\to Q$ of involutive quantales as in~\cite{MP1,OMIQ}, in analogy with locale theory. In~\cite{QFB} the relation between quantales and Fell bundles is mediated  by maps $p:\Max A\to\opens(G)$, where $G$ is a suitable topological \'etale groupoid.

We have tried to make the presentation of this paper self contained but also streamlined, and it would be useful for the reader to be familiar with the interplay between \'etale groupoids, inverse quantal frames and pseudogroups, as introduced in~\cite{Re07}. Other basic facts, terminology and notation concerning quantales, groupoids, and inverse semigroups, can be found in~\cite{Rosenthal1,gamap2006,Lawson}. By a pseudogroup (instead of abstract complete pseudogroup as in~\cite{Re07}) will be meant a complete and infinitely distributive inverse semigroup, as in~\cite{LL}.

This paper is partly based on results that were presented in a workshop at Imperial College, London, in 2009~\cite{ICOx}. The kind invitation and support of A.\ D\"oring during that visit is gratefully acknowledged.

\section{Groupoids from projections}\label{sec:stabgelf}

\paragraph{Projections and pseudogroups.}
Let $Q$ be an involutive quantale, and let $b\in Q$.
An element $s\in Q$ is a \emph{partial unit relative to $b$} if it satisfies the following conditions:
\begin{enumerate}
\item $s^*s\le b$
\item $ss^*\le b$
\item $sb\le s$
\item $bs\le s$
\end{enumerate}
The set of partial units of $Q$ relative to $b$ is denoted by $\ipi_b(Q)$. We note that if $b$ is the top element $1$ then $\ipi_b(Q)$ coincides with the set  $\ts(Q)$ of two-sided elements.

In most of this paper the element $b$ with respect to which $\ipi_b(Q)$ is defined will be assumed to be a projection, but the following fact holds without this assumption:

\begin{lemma}
Let $Q$ be an involutive quantale, and let $b\in Q$.
Then $\ipi_b(Q)$ is closed under meets of nonempty subsets in $Q$.
\end{lemma}

\begin{proof}
Let $S\subset \ipi_b(Q)$. Then we have
\[
\bigl(\bigwedge S\bigr)b\le\bigwedge_{s\in S} sb\le \bigwedge_{s\in S} s=\bigwedge S\;,
\]
and, similarly, $b\bigl(\bigwedge S\bigr)\le \bigwedge S$. Assuming now that $S\neq\emptyset$ and choosing a fixed element $s\in S$ we have
\[
\bigl(\bigwedge S\bigr)^*\bigl(\bigwedge S\bigr)\le s^* s\le b\;,
\]
and, similarly, $\big(\bigwedge S\bigr)\bigl(\bigwedge S\bigr)^*\le b$. Hence, $\bigwedge S\in\ipi_b(Q)$. \qedhere
\end{proof}

\begin{lemma}\label{twosidedlemma}
Let $Q$ be a stably Gelfand quantale. Then $Q$ is a Gelfand quantale, so the involutive subquantale $\ts(Q)$ of two-sided elements of $Q$ is idempotent with trivial involution, and unital with $e=1$ (hence being a locale with multiplication equal to $\wedge$~\cite{JT}*{Prop.\ III.1.1}).
\end{lemma}

\begin{proof}
Let $a\in Q$ be a right-sided element (\ie, $a1\le a$). Then $aa^*a\le a1\le a$, and thus $aa^*a=a$ because we are assuming that $Q$ is stably Gelfand. This shows that $Q$ is a Gelfand quantale, and the remaining assertions follow from~\cite{MP1}*{section~2}. \qedhere
\end{proof}

\begin{theorem}\label{thm:IbQ}
Let $Q$ be a stably Gelfand quantale, and let $b$ be a projection.
\begin{enumerate}
\item $\ipi_b(Q)$ is a pseudogroup, which will be referred to simply as the \emph{pseudogroup of $b$}.
\item The idempotents of $\ipi_b(Q)$ coincide with the two-sided elements of the subquantale $\downsegment b$.
\end{enumerate}
\end{theorem}

\begin{proof}
Let $s\in\ipi_b(Q)$. Then
$ss^*s\le sb\le s$,
and thus $s=sb=ss^*s$. Similarly, $s=bs$. In addition, we show that $\ipi_b(Q)$ is closed under multiplication, and thus that it is a subsemigroup of $Q$ (and also a monoid with unit $b$ because $b\in\ipi(Q)$): indeed, for all $s,t\in\ipi_b(Q)$ we have
\[\begin{array}{ccccccccc}
(st)^*st&=&t^*s^*st&\le& t^*bt&\le& t^*t&\le& b\\
st(st)^*&=&stt^*s^*&\le& sbs^*&\le& ss^*&\le& b\\
(st)b&=&s(tb)&\le& st\\
b(st)&=&(bs)t&\le& st\;.
\end{array}\]
We note that
$\downsegment b$ is a stably Gelfand subquantale of $Q$. If $s\in\ts(\downsegment b)$ we have $ss^*\le bb^*=b$, $s^*s\le b^*b=b$, and $sb\le s$ and $bs\le s$, so that $\ts(\downsegment b)\subset\ipi_b(Q)$. Moreover, by Lemma~\ref{twosidedlemma}, $\ts(\downsegment b)$ is a locale (with $\wedge$ equal to multiplication), and thus it is contained in the set of idempotents of $\ipi_b(Q)$:
\[\ts(\downsegment b)\subset E(\ipi_b(Q))\;.\]
Conversely, if $f$ is an idempotent of $\ipi_b(Q)$ we have $fb\le f$ and $bf\le f$ and also
\[f=ff^*f=ff^*f^*f\le ff^*b\le ff^*\le b\;,\]
and thus $f\in\ts(\downsegment b)$.
This shows that the set of idempotents of $\ipi_b(Q)$ is a locale. Therefore the idempotents commute, and thus $\ipi_b(Q)$ is an inverse semigroup (because it also satisfies $ss^*s=s$). Moreover, it is infinitely distributive because the subsemigroup of idempotents is.

Finally let us see that $\ipi_b(Q)$ is complete. Let $S\subset \ipi_b(Q)$ be a compatible subset and let $s=\V S$ in $Q$. We have
\begin{eqnarray*}
sb&=&\V\{tb\st t\in S\}\le s\\
bs&=&\V\{bt\st t\in S\}\le s\\
ss^*&=&\V\{tu^*\st t,u\in S\}\le b\\
s^*s&=&\V\{u^* t\st t,u\in S\}\le b\;,
\end{eqnarray*}
where the latter two equations use the fact that $S$ is a compatible set, so $s\in\ipi_b(Q)$.
\end{proof}

\paragraph{Groupoids.} There is a bijection up to isomorphisms between pseudogroups and (localic) \'etale groupoids~\cite{Re07}. For each \'etale groupoid $G$ the bijection gives us the pseudogroup of local bisections $\ipi(G)$. The converse construction is mediated by quantales (see section~\ref{sec:locproj}) and it is such that from a pseudogroup $S$ we obtain an \'etale groupoid $G$ whose locale of objects $G_0$ is isomorphic to the locale of idempotents $E(S)$, and whose locale of arrows $G_1$ is the locale $\lcc(S)$ of compatible ideals of $S$, where by a compatible ideal is meant a downwards closed set $J\subset S$ such that $\V Y\in J$ for all the (possibly empty) compatible subsets $Y\subset J$ (in particular $0\in J$ because $\emptyset\subset J$, so the compatible ideals are nonempty). 

Let $b$ be a projection of a stably Gelfand quantale $Q$. The localic \'etale groupoid that corresponds to the pseudogroup $\ipi_b(Q)$ will be denoted by $\groupoid_b(Q)$. We have:
\begin{eqnarray}
\ipi(\groupoid_b(Q))&\cong&\ipi_b(Q)\;;\\
\groupoid_b(Q)_0&\cong&E(\ipi_b(Q))=\ts(\downsegment(b))\;.
\end{eqnarray}
We say that the projection $b$ is \emph{spatial} if the locale $\ts(\downsegment(b))$ is spatial. In this case the locale $\groupoid_b(Q)_0$ is spatial, and thus so is $\groupoid_b(Q)_1$~\cite{elephant}*{Lemma 1.3.2(v)}, so $\groupoid_b(Q)$ can be regarded as a topological groupoid. Then the locale $\lcc(\ipi_b(Q))$ is order-isomorphic to the topology of (the space of arrows of) the groupoid $\groupoid_b(Q)$, with quantale operations obtained pointwise from the groupoid operations.

\section{Examples}\label{sec:examples}

\paragraph{C*-algebras.}

If $A$ is a C*-algebra and $B$ is a sub-C*-algebra of $A$, the locale $\ts(\downsegment(B))$ is precisely the locale $I(B)$ of closed two-sided ideals of $B$, which is isomorphic to the Jacobson topology of the primitive spectrum of $B$, and thus $B$ is a spatial projection of $\Max A$. Hence, there is a topological \'etale groupoid $\groupoid_B(\Max A)$ associated to each sub-C*-algebra $B\subset A$. This is the motivating example in this paper.

If $B$ is a commutative subalgebra then the object space of this groupoid is a locally compact Hausdorff space (the spectrum of $B$), and thus the groupoid is locally compact and locally Hausdorff. Moreover, if $A$ is separable and $B$ is a Cartan subalgebra in the sense of~\cite{Renault}, it follows from \cite{Renault}*{Theorem 5.9} and \cite{QFB}*{Theorem 6.8 and Theorem 6.10} that $\groupoid_B(\Max A)$ is a second countable topologically principal locally compact Hausdorff \'etale groupoid, isomorphic to the Weyl groupoid of $B$ in the sense of~\cite{Renault}.

\paragraph{Quantales of binary relations.}

Let $X$ be a set, and let $Q=2^{X\times X}$ be the unital involutive quantale of binary relations on $X$. The order is inclusion, the multiplication unit is the diagonal relation $\Delta_X$, and the multiplication of two relations $R,S\subset X\times X$ is
\[
RS=\bigl\{ (z,x)\in X\times X\st \exists_{y\in X}\ (z,y)\in R\textrm{ and }(y,x)\in S\bigr\}\;.
\]
The pseudogroup $\ipi(Q)$ consists of the (graphs of) partial bijections on $X$; that is, bijections $\alpha:Y\to Z$ for subsets $Y,Z\subset X$.
Quantales of binary relations are strongly Gelfand, and a projection in $Q$ is the same as an equivalence relation $R\subset Y\times Y$ on a subset $Y\subset X$. Then the pseudogroup $\ipi_R(Q)$ is isomorphic to $\ipi(Q')$, where $Q'$ is the quantale of binary relations on the quotient set $Y/R$: 
\[
Q'=2^{Y/R\times Y/R}\;.
\]
To each $U\in\ipi_R(Q)$ the isomorphism assigns the partial bijection $\boldsymbol\alpha_U$ on $Y/R$ whose domain is the set
\[
\dom(\boldsymbol\alpha_U) = \bigl\{[y]\st (y',y)\in U\textrm{ for some }y'\in U\bigr\}\;,
\]
and which is such that, for each $[y]\in \dom(\boldsymbol\alpha_U)$,
\[
\boldsymbol\alpha_U([y]) = \{y'\in Y\st (y',y)\in U\}\;.
\]
Conversely, given a bijection $\alpha:Z\to W$ with $Z,W\subset Y/R$,
define $\boldsymbol U_\alpha\in \ipi_R(Q)$ by
\[
\boldsymbol U_\alpha = \bigl\{(y',y)\in X\times X\st [y]\in Z\textrm{ and }\alpha([y])=[y']\bigr\}\;.
\]
Then we can check that $\boldsymbol\alpha_{\boldsymbol U_\alpha}=\alpha$ and $\boldsymbol U_{\boldsymbol \alpha_U}=U$, and that $\boldsymbol \alpha_{(-)}$ is a homomorphism of pseudogroups, so
\[
\ipi_R(Q)\cong\ipi(Q')\;.
\]
In terms of groupoids, $\groupoid_R(Q)$ is (isomorphic to) the pair groupoid (= total binary relation) on $Y/R$.

Two subexamples are worth mentioning:
\begin{itemize}
\item If $Y\subset X$ and $R=\Delta_Y\subset \Delta_X$ then $\groupoid_R(Q)$ is the pair groupoid on $Y$.
\item If $R=X\times X$ then $\groupoid_R(Q)$ is the groupoid $\boldsymbol 1$ with a single object and its identity arrow. Accordingly, the pseudogroup $\ipi_R(Q)$ is the two element chain $\ts(Q)=\{\emptyset,X\times X\}$, which is isomorphic to the discrete topology of $\boldsymbol 1$.
\end{itemize}

The latter two examples are easy to describe in the more general situation where $G$ is an arbitrary topological \'etale groupoid with object space $G_0$, and $Q=\opens(G)$ is its topology:
\begin{itemize}
\item If $Y\subset G_0$ is an open set of objects then $\groupoid_Y(Q)$ is (isomorphic to) the full subgroupoid of $G$ on $Y$ (its arrows are all the arrows of $G$ with domain and codomain in $Y$).
\item And $\groupoid_G(Q)$ is the space $G_0/G$ of orbits of $G$, regarded as a topological groupoid $H=H_0$.
\end{itemize}

\section{Localic projections}\label{sec:locproj}

\paragraph{Inverse quantal frames.} Let us recall a few more facts about the relation between pseudogroups and \'etale groupoids, now taking into account inverse quantal frames~\cite{Re07}.
Let $S$ be a pseudogroup, and let $Q$ be an involutive quantale. By a \emph{homomorphism} $\varphi:S\to Q$ will be meant a mapping that satisfies the following conditions for all $s,t\in S$ and all compatible sets $Y\subset S$:
\[
\varphi(st)=\varphi(s)\varphi(t)\;,\quad\varphi(s^{-1})=\varphi(s)^*\;,\quad\varphi\bigl(\V Y\bigr)=\V\varphi(Y)\;.
\]
The locale $\lcc(S)$ described in the previous section also has a structure of unital involutive quantale (an inverse quantal frame), and we refer to it as the \emph{quantale completion} of $S$. The involutive quantale structure of $\lcc(S)$ is easy to describe: the involution is computed by taking pointwise inverses, the order is inclusion, and the multiplication $JK$ of two compatible ideals $J$ and $K$ is the least compatible ideal that contains their pointwise product. We also note that binary meets $J\wedge K$ coincide with intersections $J\cap K$. Moreover, the principal order ideals \[\downsegment(s)=\{t\in S\st t\le s \}\] are compatible ideals, the mapping $\varphi_S:S\to\lcc(S)$ defined by $s\mapsto\downsegment(s)$ is a homomorphism, and it has the following universal property: for all involutive quantales and all homomorphisms $\varphi:S\to Q$ there is a unique homomorphism of involutive quantales
$\varphi^\sharp:\lcc(S)\to Q$ such that the following diagram commutes:
\[
\xymatrix{
S\ar[rr]^{\varphi_S}\ar[rrd]_{\varphi}&&\lcc(S)\ar[d]^{\varphi^\sharp}\\
&&Q
}
\]
Note that $\varphi^\sharp$ is defined, for all compatible ideals $J$, by
\[
\varphi^\sharp(J)=\V \varphi(J)\;.
\]
Moreover, $\lcc(S)$ is an inverse quantal frame, and every inverse quantal frame arises like this, up to isomorphisms, for if $Q$ is an inverse quantal frame we have $Q\cong\lcc(\ipi(Q))$. In particular, $\lcc(S)$ is a unital quantale, whose multiplicative unit is $e=E(S)$, and its order has the locale distributivity property:
\[
J\cap \V_i K_i = \V_i J\cap K_i\;.
\]

From any inverse quantale frame $Q$ there is a construction of an \'etale groupoid $\groupoid(Q)$ and, conversely, from any \'etale groupoid $G$ one obtains an inverse quantal frame $\opens(G)$. If $G$ is a spatial groupoid it can be identified with a topological groupoid, and $\opens(G)$ can be identified with the topology  of $G$ (cf.\ section~\ref{sec:examples}). The $\groupoid$ and $\opens$ constructions satisfy the following properties,
\begin{eqnarray*}
G&\cong&\groupoid(\opens(G))\\
Q&=&\opens(\groupoid(Q))\\
\ipi(G)&\cong&\ipi(\opens(G))\;,
\end{eqnarray*}
so in particular there is a bijection between isomorphism classes of \'etale groupoids and of inverse quantal frames.
The groupoid $G$ which is associated to a pseudogroup $S$ is defined in this way from the inverse quantal frame $\lcc(S)$,
\[
G:=\groupoid(\lcc(S))\;,
\]
so we have
\begin{equation}\label{opvslcc}
\opens(G)=\lcc(S)\;.
\end{equation}

\paragraph{Homomorphisms and maps from projections.} Let us return to stably Gelfand quantales. Let $Q$ be a stably Gelfand quantale, and let $b\in Q$ be a projection. By~\eqref{opvslcc}, the groupoid and the pseudogroup associated to $b$ are related via their inverse quantal frames by
\begin{equation}\label{opvslcc2}
\opens(\groupoid_b(Q))=\lcc(\ipi_b(Q))\;,
\end{equation}
so we shall write either $\lcc(\ipi_b(Q))$ or $\opens(\groupoid_b(Q))$ somewhat arbitrarily, or according to notational convenience.

The inclusion $\ipi_b(Q)\to Q$ is a homomorphism, and thus it has an extension to a homomorphism of involutive quantales
\[
\varphi_b:\lcc(\ipi_b(Q))\to Q\;,
\]
given, for all $J\in\lcc(\ipi_b(Q))$, by
\[
\varphi_b(J) = \V J\;.
\]
We also denote by
\[
p_b:Q\to\opens(\groupoid_b(Q))
\]
the map of involutive quantales that is defined by the condition $p_b^*=\varphi_b$.

The image of $\varphi_b$ is an involutive subquantale of $Q$, which we denote by $\opens_b(Q)$. We say that the projection $b$ is \emph{localic} if $\opens_b(Q)$ is itself a locale (\ie, an involutive quantal frame).

\begin{example}\label{em:invquantales}
Let $Q$ be an inverse quantale~\cite{Re07}. Then $\ipi_e(Q)=\ipi(Q)$, $\opens_e(Q)=Q$, and $\varphi_e:\lcc(\ipi(Q))\to Q$ is a surjective homomorphism. It follows that $\varphi_e$ is injective, and $e$ is a localic projection, if and only if $Q$ is an inverse quantal frame.
\end{example}

\begin{theorem}\label{littlelemma1}
Let $Q$ be a stably Gelfand quantale, and let $b\in Q$ be a projection. The following conditions are equivalent:
\begin{enumerate}
\item $b$ is localic;
\item $\varphi_b$ is injective;
\item $p_b$ is a surjection.
\end{enumerate}
\end{theorem}

\begin{proof}
Assume that 1 holds, and let $I,J\in\lcc(\ipi_b(Q))$. Since we are not assuming that $\opens_b(Q)$ is closed under binary meets in $Q$, let us write $\sqcap$ for the binary meet operation in $\opens_b(Q)$. We have, for all $I,J\in\lcc(\ipi_b(Q))$,
\[
\varphi_b(I)\sqcap \varphi_b(J) = \bigl(\V I\bigr)\sqcap\bigl(\V J\bigr) = \V_{a\in I,b\in J} a\sqcap b\;.
\]
Moreover, the binary meet operation in $\ipi_b(Q)$ coincides with that of $Q$, so we have
\[
\varphi_b(I)\sqcap \varphi_b(J) = \V_{a\in I,b\in J} a\wedge b\;.
\]
Since compatibly closed ideals $I$ and $J$ are downwards closed in $\ipi_b(Q)$, and $\ipi_b(Q)$ is closed under binary meets in $Q$, we further have
\[
I\cap J = \{a\wedge b\st a\in I,\ b\in J\}\;,
\]
and thus we obtain
\[
\varphi_b(I)\sqcap \varphi_b(J) = \V (I\cap J) = \varphi_b(I\cap J)\;.
\]
Hence, $\varphi_b$ defines a homomorphism of locales onto $\opens_b(Q)$. Since $\lcc(\ipi_b(Q))$ has a downwards closed basis consisting of the principal order ideals in $\ipi_b(Q)$, and $\varphi_b$ is injective on the basis, it follows that it is injective \cite{Re07}*{Prop.\ 2.2}; that is, 2 holds. Now assume that 2 holds. Then $\varphi_b$ defines an order isomorphism onto its image, which therefore is a locale.
Finally, 2 and 3 are equivalent by definition, for a surjective map of involutive quantales $p$ is defined to be one whose inverse image homomorphism $p^*$ is injective.
\end{proof}

\section{Inverse quantal frames revisited}

Recall that by an involutive quantal frame $Q$ is meant an involutive quantale whose order has the locale distributivity property: for all $a\in Q$ and all $Y\subset Q$ we have
\[
a\wedge\V Y = \V_{y\in Y} a\wedge y\;.
\]
An example is provided by inverse quantal frames, which in~\cite{Re07} are shown to be precisely the unital involutive quantal frames $Q$ that satisfy the covering condition \eqref{cc} and for which the following two conditions hold for all $a\in Q$ ($a1\wedge e$ is the support of $a$):
\begin{eqnarray}
a1\wedge e&\le&aa^*\label{cond1}\\
a&\le&(a1\wedge e)a\label{cond2}\;.
\end{eqnarray}
Any unital involutive quantale satisfying these two conditions is obviously strongly Gelfand and, hence, stably Gelfand. In what follows we shall see that for unital involutive quantal frames satisfying \eqref{cc} the converse holds:  \eqref{cond1}--\eqref{cond2} can be replaced by the simpler assumption that $Q$ is stably Gelfand.

\begin{theorem}\label{thm:invqfr}
Let $Q$ be a stably Gelfand quantal frame, and let $b\in Q$ be a projection. Then $\varphi_b$ is injective (equiv., $p_b$ is a surjection) and it preserves binary meets. In addition, if $\V\ipi_b(Q)=1$ then $\varphi_b:\lcc(\ipi_b(Q))\to Q$ is a homomorphism of locales (equiv., $p_b$ is a map of locales).
\end{theorem}

\begin{proof}
$\ipi_b(Q)$ is a pseudogroup closed under nonempty meets in $Q$, and for all $J,K\in\lcc(\ipi_b(Q))$ we have
\begin{eqnarray*}
\V(J\cap K)&=&\varphi_b(J\cap K) \le \varphi_b(J)\wedge \varphi_b(K)\\
& =& \bigl(\V J\bigr)\wedge\bigl(\V K\bigr)
=\V_{s\in J,t\in K} s\wedge t = \V(J\cap K)\;,
\end{eqnarray*}
where the last equality follows from $J$ and $K$ being downwards closed. Hence, we see that $\varphi_b$ preserves binary meets, and thus its image $\opens_b(Q)$ is a locale, which by Theorem~\ref{littlelemma1} also implies that $\varphi_b$ is injective. If $\V\ipi_b(Q)=1$ it follows that $\varphi_b(1)=1$, and thus $\varphi_b$ is a homomorphism of locales as stated.
\end{proof}

\begin{corollary}\label{cor:iqfs}
Let $Q$ be a stably Gelfand unital quantal frame such that $\V\ipi(Q)=1$. Then $Q$ is an inverse quantal frame.
\end{corollary}

\begin{proof}
Applying Theorem~\ref{thm:invqfr} with $b=e$, we have that $\varphi_e:\lcc(\ipi(Q))\to Q$ is an injective homomorphism of locales. Moreover, for all $a\in Q$ we have
\[
a=a\wedge 1=a\wedge\V\ipi(Q) = \V_{s\in\ipi(Q)} a\wedge s\;.
\]
Since $a\wedge s$ is in $\ipi(Q)$ for all $s\in\ipi(Q)$,
we conclude that $\ipi(Q)$ is join-dense in $Q$, and thus $\varphi_e$ is surjective, showing that $\varphi_e$ is an isomorphism and, therefore, that $Q$ is an inverse quantal frame isomorphic to $\lcc(\ipi(Q))$.
\end{proof}

\section{Projections versus maps}\label{sec:projmaps}

Let $Q$ be a stably Gelfand quantale. As we have seen, a projection $b\in Q$ determines a map $p_b:Q\to\opens(\groupoid_b(Q))$.
We regard $p_b$ as a ``quantal bundle'' over the groupoid $\groupoid_b(Q)$.
Conversely, given an \'etale groupoid $G$ and a map $p:Q\to\opens(G)$ we obtain a projection $p^*(e)$ in $Q$, so there is a back and forth correspondence between projections and ``bundles''. Let us give a brief account of this. First we note that the construction of $p_b$ relates canonically to other maps $q:Q\to\opens(G)$ via a \emph{comparison map} $k$, as follows:

\begin{theorem}\label{lem:relatingquantalbundles}
Let $G$ be a localic \'etale groupoid, and $q:Q\to\opens(G)$ a map. Then $b:=q^*(e)$ is a projection of $Q$, and there is a unique map of unital involutive quantales $k:\opens(\groupoid_b(Q))\to\opens(G)$ that makes the following diagram commute, which moreover is a surjection if $q$ is:
\begin{equation}\label{diagramcompmap}
\xymatrix{
&Q\ar[dl]_{p_b}\ar[dr]^{q}\\
\opens(\groupoid_b(Q))\ar[rr]_-{k}&&\opens(G)
}
\end{equation}
\end{theorem}

\begin{proof}
The multiplication unit $e\in\opens(G)$ is a projection of $\opens(G)$, and thus $p_b^*(e)$ is a projection of $Q$ because $p_b^*$ is an involutive homomorphism. Let $s\in\ipi(\opens(G))$. Then we have $p_b^*(s)b=p_b^*(s)p_b^*(e)=p_b^*(se)=p_b^*(s)$ and $p_b^*(s)^*p_b^*(s)=p_b^*(s^*)p_b^*(s)=p_b^*(s^*s)\le p_b^*(e)=b$, and, similarly, we obtain $bp_b^*(s)=p_b^*(s)$ and $p_b^*(s)p_b^*(s)^*\le b$, showing that $p_b^*(s)\in\ipi_b(Q)$. Hence, $p_b^*(\ipi(\opens(G)))\subset\ipi_b(Q)$ and thus $p_b^*$ restricts to a homomorphism of pseudogroups $\ipi(\opens(G))\to\ipi_b(Q)$. Applying the functor $\lcc$ we extend this to a homomorphism of unital involutive quantales $h:\lcc(\ipi(\opens(G)))\to\lcc(\ipi_b(Q))$, and the inverse image homomorphism of $k$ is defined as the following composition:
\[
\xymatrix{
\opens(G)\ar@/_1cm/[rrr]_{k^*}\ar[r]^-{\cong}&\lcc(\ipi(\opens(G)))\ar[r]^-{h}&\lcc(\ipi_b(Q))\ar@{=}[r]&\opens(\groupoid_b(Q))
}
\]
Of course, $k$ makes the diagram \eqref{diagramcompmap} commute because the restrictions of the three maps to the pseudogroups $\ipi(\opens(G))$ and $\ipi_b(Q)$ commute, and it is clear that $k$ is the only map with this property. Moreover, if $q$ is a surjection then $h$ is injective, and thus $k$ is a surjection.
\end{proof}

\begin{corollary}
Let $\boldsymbol Q$ be the category whose objects are the pairs $(Q,b)$ where $Q$ is a stably Gelfand quantale and $b\in Q$ is a projection, and whose morphisms $p:(Q,b)\to(R,c)$ are the maps $p:Q\to R$ such that $p^*(c)=b$. And let $\boldsymbol I$ be the category whose objects are the inverse quantal frames and whose morphisms $p:Q\to R$ are the maps of unital involutive quantales. The functor $U:\boldsymbol I\to\boldsymbol Q$ defined on objects by $U(Q)=(Q,e)$ has a left adjoint that to each pair $(Q,b)$ assigns $\opens(\groupoid_b(Q))$, and the adjunction is a reflection.
\end{corollary}

\begin{proof}
The unit of the adjunction is the family of maps
\[
p_b:(Q,b)\to(\opens(\groupoid_b(Q)),e)
\]
and the universal property follows from the uniqueness of the comparison map $k$ in Theorem~\ref{lem:relatingquantalbundles}. 
Let $F$ be the left adjoint. For each inverse quantal frame $Q$ the $Q$-component of the counit of the adjunction is an isomorphism
\[
\xymatrix{
F(U(Q))=\opens(\groupoid_e(Q))=\lcc(\ipi_e(Q))=\lcc(\ipi(Q))\ar[rr]^-{\cong}&& Q\;,
}
\]
so the adjunction is a reflection.
\end{proof}

One advantage of focusing on maps of involutive quantales instead of homomorphisms is the possibility of applying topological language to maps, such as when studying notions akin to that of open map of locales. For instance, a map of involutive quantales $p:Q\to X$ is said to be semiopen if $p^*$ has a left adjoint $p_!$, which we refer to as the direct image homomorphism of $p$. A stronger notion is that of open map of~\cite{OMIQ}, of which a class of examples is provided by the surjections of involutive quantales $p:Q\to X$, with $X$ unital, that satisfy the two-sided Frobenius reciprocity condition for all $a,a'\in Q$ and $x\in X$:
\[
p_!(ap^*(x)a') = p_!(a)xp_!(a')\;.
\]
If $X=\opens(G)$ for some \'etale groupoid $G$ there is more that we can say about the comparison map of Theorem~\ref{lem:relatingquantalbundles}.

\begin{theorem}\label{thm:split}
Let $Q$ be a stably Gelfand quantale, and $G$ an \'etale groupoid. Let also $q:Q\to\opens(G)$ be a semiopen surjection satisfying the two-sided Frobenius reciprocity condition, and let $b$ be the projection $q^*(e)\in Q$. The comparison map $k:\opens(\groupoid_b(Q))\to\opens(G)$ is a split surjection.
\end{theorem}

\begin{proof}
The fact that $q$ is a surjection implies that the comparison map $k$ is a surjection, by Theorem~\ref{lem:relatingquantalbundles}. In order to see that $k$ splits let $v\in\ipi_b(Q)$ and $s=q_!(v)$. Then
\begin{equation}\label{eq:prepshrhom}
s^*s=q_!(v^*)q_!(v)\ge q_!(v^*v)\;,
\end{equation}
due to adjointness, because $q^*$ is a homomorphism. But we also have, using the two-sided Frobenius reciprocity condition,
\begin{equation}\label{eq:pshrhom}
s^*s=s^*es=q_!(v^*)eq_!(v)=q_!(v^*q^*(e)v)=q_!(v^*bv)\le q_!(v^*v)\le q_!(b)\;,
\end{equation}
and, again using the fact that $q$ is a surjection, we have $q_!(b)=q_!(q^*(e))=e$, so $s^*s\le e$. Analogously we prove that $ss^*\le e$, and thus $s\in\ipi(\opens(G))$. So $q_!$ restricts to a mapping $f:\ipi_b(Q)\to\ipi(\opens(G))$, which by \eqref{eq:prepshrhom}--\eqref{eq:pshrhom} is a homomorphism of pseudogroups, and the section that splits $k$ is the map $\sigma$ whose inverse image homomorphism is given by the following composition:
\[
\xymatrix{
\opens(G)\ar@{<-}@/_1cm/[rrrr]_{\sigma^*}\ar@{<-}[r]^-{\cong}&\lcc(\ipi(\opens(G)))\ar@{<-}[rr]^-{\lcc(f)}&&\lcc(\ipi_b(Q))\ar@{<-}@{=}[r]&\opens(\groupoid_b(Q))
}
\]
\end{proof}

\vspace*{5mm}
\noindent {\sc
Centro de An\'alise Matem\'atica, Geometria e Sistemas Din\^amicos
Departamento de Matem\'{a}tica, Instituto Superior T\'{e}cnico\\
Universidade de Lisboa\\
Av.\ Rovisco Pais 1, 1049-001 Lisboa, Portugal}\\
{\it E-mail:} {\sf pmr@math.tecnico.ulisboa.pt}

\end{document}